\documentclass[11pt]{amsart}

\usepackage{a4wide,enumerate,color,graphicx}
\usepackage{amsmath,bm}
\usepackage{scrextend} 
\allowdisplaybreaks
\usepackage[normalem]{ulem}
\usepackage[dvipsnames]{xcolor}
\usepackage{setspace}
\usepackage{amsbsy}
\usepackage{hyperref}
\usepackage[overload]{empheq}

\newcommand{\R}{{\mathbb R}} 
\newcommand{\del}{\partial}
\newcommand{\diver}{\operatorname{div}} 
\newcommand{\cc}{\boldsymbol{c}}
\newcommand{\dx}{\textnormal{d}x}
\newcommand{\dt}{\textnormal{d}t}
\newcommand{\dtau}{\textnormal{d}\tau}

\numberwithin{equation}{section}
\hypersetup{
    colorlinks=true,
    linkcolor=blue,
    filecolor=blue,      
    urlcolor=blue,
}

\newtheorem{theorem}{Theorem}
\newtheorem{lemma}[theorem]{Lemma}

\newtheorem{remark}[theorem]{Remark}

\newtheorem{definition}{Definition}

\begin{document}

\title[Absence of anomalous dissipation]{Absence of anomalous dissipation for weak solutions of the Maxwell--Stefan system}
\author{Luigi C. Berselli, Stefanos Georgiadis and Athanasios E. Tzavaras} 
\address{Luigi C. Berselli, Dipartimento di Matematica, Universit\`a di Pisa, Via F. Buonarroti 1/c, I56127, Pisa, Italy }
\email{luigi.carlo.berselli@unipi.it}
\address{Stefanos Georgiadis, Computer, Electrical and Mathematical Science and Engineering Division, King Abdullah University of Science and Technology (KAUST), Thuwal 23955-6900, Saudi Arabia and Institute for Analysis and Scientific Computing, Vienna University of Technology, Wiedner Hauptstra\ss{}e  8-10, 1040 Wien, Austria}
\email{stefanos.georgiadis@kaust.edu.sa}
\address{Athanasios E. Tzavaras, Computer, Electrical and Mathematical Science and Engineering Division, King Abdullah University of Science and Technology (KAUST), Thuwal 23955-6900, Saudi Arabia}
\email{athanasios.tzavaras@kaust.edu.sa}
\date{\today}
\begin{abstract}
    In this paper we give a short and self-contained proof of the fact that weak solutions to the Maxwell-Stefan system automatically satisfy an entropy equality, establishing the absence of anomalous dissipation.
\end{abstract}

\maketitle
\section{Introduction}

Energy balance is an important tool in the understanding of fluid systems. Already from the work of Kolmogorov in the 1940s, it became evident that the evolution of energy in fluids plays a crucial role in the mechanism of creation of turbulence. To this extent, establishing an energy equality implies that the flow is not chaotic, as there is no turbulence (with some caveat due to the variety of phenomena related with fluid flows).  It was formalized in the famous Onsager conjecture, stating that H\"older continuous solutions of the incompessible Euler equations flows with exponent $\alpha > 1/3$ conserve the kinetic energy.
The positive part of the conjecture was proved in  Constantin, E, and
Titi~\cite{CoETi94}, according to which weak solutions to the Euler
equations with some additional regularity conserve their energy. The
negative part of the conjecture was settled in an important work by
Isett \cite{Isett18} and Buckmaster \textit{et al.} \cite{BDLSV19} building on several previous works (see references therein).
Following the work \cite{CoETi94},  several studies investigated regularity criteria under which an energy conservation (or dissipation)
identity can be attained, including  the incompressible~\cite{DuRo00} or compressible~\cite{Yu17} Navier-Stokes equations, the system of magnetohydrodynamics~\cite{CaKlSt97}, the Euler-Korteweg system~\cite{DeGwSwTz18} and general conservation laws~\cite{GwMiSw18}. See \cite{Ber2021} for a review of  results on incompressible flows under various regularity assumptions that induce conservation of energy.

The study of multicomponent systems is an active research field of mathematical fluid mechanics for compressible flows. 
The Maxwell-Stefan system,~ see \eqref{equation} below,  is one of the simplest models describing the diffusive transport
of the components of a gaseous mixture with zero mean flow.  
It consists of a system of continuity equations, one for the mass fraction of each component, together with
an algebraic system describing frictional balance among the components, and coupling the  molar fluxes with the mass fractions.
The system is equipped with an entropy dissipation structure 
obtained from thermodynamic considerations which plays an important role in its analysis.
It induces a gradient flow structure \cite{HuJueTz22} which renders it formally a parabolic system. 
Local existence of smooth solutions is shown in \cite{Bothe} while global existence of weak solutions is shown in \cite{JueSt13}. At present it is not
known if smooth solutions break down and in what norm. 

The solutions constructed in \cite{JueSt13} inherit from construction an {\it energy dissipation inequality}.
The objective of the present work is to show that weak solutions to the Maxwell-Stefan system satisfy an {\it energy dissipation identity}, 
making a first qualitative connection with similar results about standard turbulence theory. The main result we prove is the absence of anomalous dissipation for weak solutions to the Maxwell-Stefan system in the same class where  we have global existence, see Theorem~\ref{theorem}. 
This property is obtained as a consequence of the regularity of the weak solutions, and of the structure of the equations, without additional assumptions. 

\subsection{Setting} We consider the system of Maxwell-Stefan equations which describes multicomponent diffusive phenomena in a gas mixture. The system consists of the partial differential equations 
\begin{equation}
    \label{equation}\del_tc_i+\diver J_i=0, \quad \textnormal{in }\Omega\times(0,T),
\end{equation} 
describing the evolution of the mass fractions $c_i$ for each
component $i=1,\dots,n$,  where  $J_i$ is the molar flux of the $i$-th
component and the evolution takes place in   a bounded domain $\Omega\subset \R^3$ with smooth boundary $\partial\Omega$.

Initially, the mass fractions are given by \begin{equation}\label{IC}
    c_i(x,0)=c_i^0(x) \quad \textnormal{ in } \Omega,
\end{equation} where $c_i^0(x)\geq0$ and $\sum_{i=1}^nc_i^0(x)=1$ and we assume that no mass can enter or leave the domain, i.e., that  
\begin{equation}\label{BC}
    J_i\cdot\nu=0, \quad \textnormal{ on } \del\Omega\times(0,T).
\end{equation} where $\nu$ is the exterior normal unit vector to $\del\Omega$.

The molar fluxes $J_i$ are solutions of the linear system

\begin{equation}\label{sys}
    -\sum_{j\not=i,j=1}^n\frac{c_jJ_i-c_iJ_j}{D_{ij}}=\nabla c_i\quad \textnormal{in }\Omega\times(0,T),
\end{equation}
\noindent
where $D_{ij}$, for $i\not=j$, are positive and symmetric coefficients
describing binary interactions between the $i$-th and $j$-th
components, and the molar fluxes are subject to the constraint
\begin{equation}\label{constr1}
    \sum_{i=1}^nJ_i=0.
\end{equation}

In the case of nonvanishing concentrations, the molar fluxes $J_i$ are
defined as the product of the mass fractions $c_i$ and the diffusional
velocities $u_i$, i.e. $J_i:=c_iu_i$. In \cite{HuJueTz22} it was shown
that system \eqref{sys}-\eqref{constr1} can be rewritten in terms of
the quantities
\[ m_i:=\sqrt{c_i}u_i,\]
which make sense even for $c_i=0$, so that $J_i$ should be understood as $J_i:=\sqrt{c_i}m_i$ and $m_i$ can be found by inverting the algebraic linear system \begin{equation}\label{alg-system}
    -\sum_{j=1}^nA_{ij}m_j=2\nabla \sqrt{c_i},
\end{equation} with \begin{equation*}
    A_{ij}=\begin{cases}
    -\frac{\sqrt{c_ic_j}}{D_{ij}}, & i\not=j
    \\
    \sum_{k\not=i}\frac{c_k}{D_{ik}}, & i=j
\end{cases},
\end{equation*} 

\noindent
subject to the constraint
\begin{equation}\label{constr2}
    \sum_{i=1}^n\sqrt{c}_im_i=0.
\end{equation}

The main difficulty of the Maxwell-Stefan system is that system \eqref{sys} (and hence also when reformulated as \eqref{alg-system}) is singular and thus yields no unique solution. Nevertheless, under the additional constraint \eqref{constr1} (equivalently \eqref{constr2}), which ensures conservation of total mass, the linear system can be inverted on its image by means of the Bott-Duffin inverse (for more details we refer to \cite{HuJueTz22}). 
%The resulting quantity $m_i$ is well-defined, since due to \cite{JueSt13}, for weak solutions, the gradients $\nabla\sqrt{c_i}$ lie in $L^2(0,T;L^2(\Omega))$. This also implies that $m_i\in L^2(0,T;L^2(\Omega))$.

To properly define the class of solution which we will handle, we define a weak solution as follows: 
\begin{definition}
\label{def:def-weak}
    We call the vector field $\cc=(c_1,\dots,c_n)$, such that $c_i\in L^2(0,T;H^1(\Omega))$ and $\partial_t c_i\in L^2(0,T;(H^1(\Omega))^*)$, a weak solution of \eqref{equation}-\eqref{sys} if:
    \begin{itemize}
        \item [(i)] it satisfies \eqref{equation} in the weak sense, i.e. \begin{equation} \label{wk-form}
            \int_0^T\int_\Omega c_i\del_t\varphi_i \,\dx\dt + \int_0^T\int_\Omega \sqrt{c_i}m_i\cdot\nabla\varphi_i \,\dx\dt=0,
        \end{equation} for all $i=1,\dots,n$ and for all $\varphi_i\in C^1_c(\overline{\Omega}\times(0,T))$, where $m_i=\sqrt{c_i}u_i$ is the solution of \eqref{alg-system}-\eqref{constr2};
        %it satisfies the initial and boundary conditions \eqref{IC}-\eqref{BC}
        \item[(ii)] $c_i\in L^\infty(0,T;L^\infty(\Omega))$, for all $i=1,\dots,n$,;
        \item [(iii)] $c_i\geq0$ with $\sum_{i=1}^nc_i=1$, a.e. in $\Omega\times[0,T]$ for all $i=1,\dots,n$,;
        \item [(iv)] $\sqrt{c_i}\in L^2(0,T;H^1(\Omega))$ and $c_i\in C([0,T];L^2(\Omega))$ for all $i=1,\dots,n$, so that the initial conditions are satisfied in the strong $L^2(\Omega)$ sense. 
        %with $\mathcal{V}=\{v\in H^2(\Omega):\nabla c\cdot\nu=0 \textnormal{ on } \del\Omega\}$ %and
    \end{itemize}
\end{definition}

\begin{remark}
    Note that from the weak formulation \eqref{wk-form} one can obtain
    the one considered in \cite{HuJueTz22}, namely \[ \int_\Omega
      c_i(T)\psi_i(T) \,\dx -\int_\Omega c_i^0\psi_i(0) \,\dx
      -\int_0^T\int_\Omega c_i\del_t\psi_i \,\dx
      \dt-\int_0^T\int_\Omega \sqrt{c_i}m_i\cdot\nabla\psi_i \,\dx
      \dt=0,
    \]
    for all $\psi_i\in C^1([0,T];C^1(\overline{\Omega}))$, by choosing in \eqref{wk-form} $\varphi_i(x,t) = \eta_\sigma(t)\psi_i(x,t)$, where $\eta_\sigma$ is the cut-off in time defined by \begin{equation}\label{cut-off}
        \eta_\sigma(t) = \begin{cases}
        \frac{t-\sigma}{\sigma}, & \sigma \leq t \leq 2\sigma \\
        1, & 2\sigma \leq t \leq T-2\sigma \\
        \frac{T-t-\sigma}{\sigma}, & T-2\sigma \leq t \leq T-\sigma \\
        0, & t\leq \sigma, ~ t\geq T-\sigma,
        \end{cases}
    \end{equation} for $0<\sigma<\frac{T}{4}$ and $\psi_i\in C^1([0,T];C^1(\overline{\Omega}))$ and then pass to the limit $\sigma \to 0$.
\end{remark}

We note, here, that since weak solutions satisfy $c_i\in L^\infty(0,T;L^\infty(\Omega))$ and $\nabla\sqrt{c_i}\in L^2(0,T;L^2(\Omega))$, the inversion of system \eqref{alg-system} yields $m_i\in L^2(0,T;L^2(\Omega))$.

The existence of such a weak solution, satisfying all the above conditions was proven in~\cite{JueSt13}.
\noindent
In addition to the properties of Definition~\ref{def:def-weak}, weak solutions satisfy an entropy inequality. Let us define the quantity 
\begin{equation}\label{entropy}
    H(\cc):=\sum_{i=1}^n\int_\Omega c_i(\ln{c_i}-1)\dx,
\end{equation} 
which is the entropy functional associated to the system. Notice that
$ H(\cc)\leq0$, since $0\leq c_i\leq1$ almost everywhere. Moreover,
the entropy is well-defined even for zero concentrations, since the
limit of the function $s\ln s$, as $s\to0^{+}$, is equal to zero. Previous works, especially \cite{JueSt13,HuJueTz22}, have shown that weak solutions, as in Definition~\ref{def:def-weak}, satisfy the following entropy inequality: 
\begin{equation}
    \label{entropy-ineq} H(\cc(T))+\frac{1}{2}\sum_{i,j=1}^n\int_0^T\int_\Omega\frac{c_ic_j}{D_{ij}}|u_i-u_j|^2 \,\dx \dt\leq H(\cc^0),
\end{equation}
where the dissipation term can be rewritten in terms of $\sqrt{c_i}$ and $m_i$ as follows: 
\[ 
\frac{1}{2}\sum_{i,j=1}^n\int_0^T\int_\Omega\frac{1}{D_{ij}}|\sqrt{c_j}m_i-\sqrt{c_i}m_j|^2\,\dx\dt,
\] and thus is well-defined, since $c_i\in L^\infty(0,T;L^\infty(\Omega))$ and $m_i\in L^2(0,T;L^2(\Omega))$. In this work, we will show that, in fact, weak solutions to the Maxwell-Stefan system as in Definition~\ref{def:def-weak} satisfy~\eqref{entropy-ineq} as an identity. It should be noted that, according to the results from~\cite{GeKiTz23}, weak solutions are in fact renormalized solutions, meaning that they further satisfy the identity \begin{equation}\label{renorm}
    \del_t\beta(c_i) + \diver(\beta'(c_i)c_iu_i) = c_iu_i\cdot\nabla\beta'(c_i),
\end{equation} in the weak sense, for all functions $\beta\in C^2([0,+\infty))$. Even though a formal computation suggests that the entropy identity could be obtained by choosing in \eqref{renorm} $\beta(s) = s\ln{s}-s$, such a $\beta$ is not admissible, due to the fact that $s\ln{s}-s$ is not a $C^1$ function up to $s=0$. To this degree, the aim of this paper is to make this formal computation rigorous.
\subsection{Main result}
The fact that weak solutions satisfy instead of~\eqref{entropy-ineq} the entropy 
equality~\eqref{entropy-eq}, is known in the literature as ``absence of anomalous dissipation'' and it is generally a basic, but additional,  assumptions used in the proofs of weak-strong uniqueness results (see Theorem~1 in~\cite{HuJueTz22}). In particular, this assumption is at the basis of the relative entropy method, which assumes that the strong solution satisfies an entropy identity, allowing for the derivation of the relative entropy inequality.

In this work we prove that \textit{all} weak solutions satisfy the entropy equality, implying that there is no anomalous dissipation. This can be interpreted in analogy with incompressible fluids as the fact that the flow is not turbulent. Since the entropy equality is ``formally'' obtained by testing the equations by $\ln{c_i}$ and integrating by parts over the domain, the essential steps will be the ``rigorous'' justification of these computations. As we will see, this is not just a matter of regularity of the solution: the fact that the solution may not be bounded away from zero (it is not enough simply to assume that $c_i>0$ since estimates involving the logarithm are not uniform near zero and one should ask the unrealistic $c_i\geq \overline{c}>0$) is a peculiar issue and makes the justification non trivial. %\tcr{We also recall that the heuristics for this approach come from the fact that the Maxwell-Stefan system can be formally seen as a generalized gradient flow of the Boltzmann entropy~\eqref{entropy} in the Wasserstein space.}

The main result of this paper is, then, the following:
\begin{theorem}\label{theorem}
    Let $\cc$ be a weak solution to the Maxwell-Stefan system \eqref{equation}-\eqref{sys} and let the initial data be such that $c_i^0(x)\geq0$ and $\sum_{i=1}^nc_i^0(x)=1$ a.e. in $\Omega$. Then $\cc$ satisfies the entropy identity 
    \begin{equation}
        \label{entropy-eq} H(\cc(t))+\frac{1}{2}\sum_{i,j=1}^n\int_0^t\int_\Omega\frac{c_ic_j}{D_{ij}}|u_i-u_j|^2 \,\dx \dt=H(\cc^0)\qquad \forall\,t\in(0,T].
    \end{equation}
\end{theorem}

%%%%%%%%%%%%%%%%%%%%%%%%%%%%%%%%
%%%%%%%%%%%%%%%%%%%%%%%%%%%%%%%%
%%%%%%%%%%%%%%%%%%%%%%%%%%%%%%%%
\section{Proof of Theorem \ref{theorem}}
We first notice that even if we use the weak formulation in terms of $\sqrt{c_i}$ and $m_i$ as in \cite{HuJueTz22}, the mass fractions $c_i$ of a weak solution satisfy the regularity stated in Definition~\ref{def:def-weak} and especially the strong $L^2$-continuity.
\begin{lemma}\label{lemma}
    Let $\cc$ be a weak solution such that $c_i\in L^\infty(0,T;L^\infty(\Omega))$ and $m_i\in L^2(0,T;L^2(\Omega))$. Then, it follows that
    \begin{equation}
        \del_tc_i\in L^2(0,T;H^1(\Omega)^*),
    \end{equation} and thus \begin{equation} \label{continuity}
        c_i\in C([0,T];L^2(\Omega)).
    \end{equation}
\end{lemma}
\begin{proof}
    Let $\langle\cdot,\cdot\rangle$ denote the duality pairing between the spaces $L^2(0,T;H^1(\Omega)^*)$ and $L^2(0,T;H^1(\Omega))$. Then, the computation 
\[ \begin{split}
         \int_0^T|\langle\del_tc_i,\varphi\rangle| \,\dt & = \int_0^T|\langle-\diver(c_iu_i),  \varphi\rangle| \,\dt \\
         & = \int_0^T|\langle c_iu_i,\nabla\varphi\rangle| \,\dt \\ 
         & \leq     \|\sqrt{c_i}\|_{L^\infty(0,T;L^\infty(\Omega))}\|m_i\|_{L^2(0,T;L^2(\Omega))}\|\nabla\varphi\|_{L^2(0,T;L^2(\Omega))},
     \end{split} 
\]
     shows that due to $m_i$ being in $L^2(0,T;L^2(\Omega))$, \[ \del_t c_i\in L^2(0,T;(H^1(\Omega)^*). \] Moreover, since 
     \[
     c_i\in L^2(0,T;H^1(\Omega)) \quad \textnormal{ and } \quad \del_tc_i\in L^2(0,T;H^1(\Omega)^*), 
     \] 
     and since
     \[ H^1(\Omega)\hookrightarrow L^2(\Omega)\hookrightarrow H^1(\Omega)^*
     \] 
     is a Gelfand evolution triple (by continuous and dense embedding), \eqref{continuity} follows  by Lions-Magenes interpolation theory.
     \end{proof}
     We can now give the proof of the main result of this paper.
\begin{proof}[Proof of Theorem~\ref{theorem}]   

In the weak formulation \eqref{wk-form} choose $\varphi_i(x,t)=(\eta_\sigma(t)\ln(c_i^\epsilon(x,t)+\delta))^\epsilon$, where $\eta_\sigma$ is the cut-off in time defined in \eqref{cut-off} for $0<\sigma<\frac{T}{4}$ and $0<\delta<1$, a constant. Moreover, for $0<\epsilon<\sigma$ we define $v^\epsilon$ to be the mollification of $v$ in time, i.e. \[ v^\epsilon(x,t):= \int_0^T\rho_\epsilon(t-\tau)v(x,\tau) \,\dtau, \] where $\rho$ is a standard symmetric mollifier. Defining mollification this way, we assume that the function $v$ is extended by zero outside $[0,T]$. This gives the usual convergence $v^\epsilon\to v$ in the interior of $[0,T]$ and eventually problems on the boundary; this is why we work with the cut-off in time. We note that admissible test functions belong to the space $C^1_c(\overline{\Omega}\times(0,T))$, but by a density argument it is enough to consider them in $H^1(0,T;H^1(\Omega))$.

We first notice that \begin{equation} \label{mol}
    \int_0^T\int_\Omega v\,\del_t \phi^\epsilon \, \dx\dt = \int_0^T\int_\Omega v^\epsilon \,\del_t\phi \,\dx\dt,
\end{equation} for $\phi$ such that $\phi(x,0)=\phi(x,T)=0$.

Indeed, we have:
\begin{align*}
    \int_0^T\int_\Omega v(x,t)\del_t \phi^\epsilon(x,t) \, \dx\dt & = \int_0^T\int_\Omega v(x,t) \left(\int_0^T \del_t\rho_\epsilon(t-\tau) \phi(x,\tau) \,\dtau \right) \, \dx\dt \\
    & = - \int_0^T\int_\Omega v(x,t) \left(\int_0^T \del_\tau \rho_\epsilon(t-\tau) \phi(x,\tau) \,\dtau \right) \, \dx\dt \\
    & = \int_0^T\int_\Omega v(x,t) \left(\int_0^T  \rho_\epsilon(t-\tau) \del_\tau\phi(x,\tau) \,\dtau \right) \, \dx\dt \\
    & = \int_0^T\int_\Omega \del_\tau\phi(x,\tau) \left(\int_0^T  \rho_\epsilon(\tau-t) v(x,t)  \,\dt \right) \, \dx\dtau \\
    & = \int_0^T\int_\Omega v^\epsilon(x,t)\del_t\phi(x,t) \, \dx\dt,
\end{align*}

\noindent
where in the second equality we used $\del_t\rho_\epsilon(t-\tau) = -\del_\tau \rho_\epsilon(t-\tau)$, in the third one an integration by parts combined with $\phi(x,0)=\phi(x,T)=0$ and in the fourth equality we used Fubini-Tonelli and the fact that $\rho$ is a symmetric mollifier.

Then, because of \eqref{mol} and since $\eta_\sigma(T)\ln(c_i^\epsilon(T)+\delta)=\eta_\sigma(0)\ln(c_i^\epsilon(0)+\delta)=0$, the first term of \eqref{wk-form} yields:

\[ \begin{split}
    \int_0^T\int_\Omega c_i\del_t\varphi_i \,\dx \dt & = \int_0^T\int_\Omega c_i^\epsilon\del_t\left[\eta_\sigma(t)\ln(c_i^\epsilon+\delta)\right] \,\dx \dt \\
    & = \int_0^T\int_\Omega c_i^\epsilon \eta_\sigma'(t)\ln(c_i^\epsilon+\delta) \,\dx \dt + \int_0^T\int_\Omega c_i^\epsilon\eta_\sigma(t)\del_t\ln(c_i^\epsilon+\delta) \,\dx \dt,
    \end{split} \]

    \noindent
    where \[ \begin{split}
        \int_0^T\int_\Omega c_i^\epsilon \eta_\sigma'(t)\ln(c_i^\epsilon+\delta) \,\dx \dt & = \frac{1}{\sigma}\int_\sigma^{2\sigma}\int_\Omega c_i^\epsilon\ln{(c_i^\epsilon+\delta)}\,\dx\dt - \frac{1}{\sigma}\int_{T-2\sigma}^{T-\sigma}\int_\Omega c_i^\epsilon\ln{(c_i^\epsilon+\delta)}\,\dx\dt,
    \end{split} \]

    \noindent
    and
    
    \[ \begin{split}
    \int_0^T \int_\Omega c_i^\epsilon \eta_\sigma(t) \del_t\ln(c_i^\epsilon+\delta) \,\dx\dt & = \int_0^T \int_\Omega c_i^\epsilon \eta_\sigma(t) \frac{\del_t c_i^\epsilon}{c_i^\epsilon+\delta} \,\dx\dt \\
    & = \int_0^T \int_\Omega \eta_\sigma(t) \del_tc_i^\epsilon \,\dx\dt -\delta \int_0^T \int_\Omega \eta_\sigma(t) \del_t\ln(c_i^\epsilon+\delta)\,\dx\dt \\
    & = -\int_0^T \int_\Omega \eta_\sigma'(t) c_i^\epsilon \,\dx\dt + \delta \int_0^T \int_\Omega \eta_\sigma'(t)\ln(c_i^\epsilon+\delta)\,\dx\dt \\
    & = - \frac{1}{\sigma}\int_\sigma^{2\sigma}\int_\Omega (c_i^\epsilon-\delta\ln{(c_i^\epsilon+\delta)}) \,\dx\dt \\
    & \phantom{xx} + \frac{1}{\sigma}\int_{T-2\sigma}^{T-\sigma}\int_\Omega (c_i^\epsilon-\delta\ln{(c_i^\epsilon+\delta)})\,\dx\dt.
\end{split} \]

The second term of \eqref{wk-form} reads \[ \begin{split}
    \int_0^T\int_\Omega c_iu_i \cdot \nabla\varphi_i \,\dx\dt & = \int_0^T\int_\Omega (c_iu_i)^\epsilon \eta_\sigma(t) \cdot \nabla\ln{(c_i^\epsilon+\delta)} \,\dx\dt \\
    & = \int_0^T\int_\Omega (c_iu_i)^\epsilon \eta_\sigma(t) \cdot \frac{\nabla c_i^\epsilon}{c_i^\epsilon+\delta} \,\dx\dt,
\end{split} \]

\noindent
and putting everything together we arrive at \begin{equation}
    \label{(1)} \begin{split}
        & \frac{1}{\sigma}\int_\sigma^{2\sigma}\int_\Omega c_i^\epsilon\ln{(c_i^\epsilon+\delta)} \,\dx\dt - \frac{1}{\sigma}\int_{T-2\sigma}^{T-\sigma}\int_\Omega c_i^\epsilon\ln{(c_i^\epsilon+\delta)} \,\dx\dt \\
        & \phantom{x} - \frac{1}{\sigma}\int_\sigma^{2\sigma}\int_\Omega (c_i^\epsilon-\delta\ln{(c_i^\epsilon+\delta)}) \,\dx\dt + \frac{1}{\sigma}\int_{T-2\sigma}^{T-\sigma}\int_\Omega (c_i^\epsilon-\delta\ln{(c_i^\epsilon+\delta)}) \,\dx\dt \\
        & \phantom{xx} + \int_0^T\int_\Omega (c_iu_i)^\epsilon \eta_\sigma(t) \cdot \frac{\nabla c_i^\epsilon}{c_i^\epsilon+\delta} \,\dx\dt = 0.
    \end{split}
\end{equation}

The plan is to first let $\epsilon\to0^+$ keeping $\sigma>0$ and $\delta>0$ fixed, then keeping $\delta>0$ fixed take the limit $\sigma\to 0^+$ and finally let $\delta\to0^+$.

\noindent
\underline{Step 1}: Keeping $\sigma>0$ and $\delta>0$ fixed, we pass to the limit $\epsilon\to0^+$. For the first two terms, due to mollification properties, $0\leq c_i \leq 1$ and $\Omega$ being bounded, we have 

\[ \begin{split}
    & c_i^\epsilon \ln{(c_i^\epsilon+\delta)} \to c_i \ln{(c_i+\delta)} ~ \textnormal{almost everywhere in } \Omega \times [0,T] \\
    & |c_i^\epsilon \ln{(c_i^\epsilon+\delta)}| \leq \max\{\ln{(1+\delta)}, |\ln{\delta}|\} \in L^1(0,T;L^1(\Omega)),
\end{split} \]
so that by the Dominated Convergence Theorem we can pass to the limit. The next two terms are treated in a similar way.

For the last term, we have
\[ 
\begin{split}
    & \int_0^T\int_\Omega (c_iu_i)^\epsilon \eta_\sigma \cdot \frac{\nabla c_i^\epsilon}{c_i^\epsilon+\delta} \,\dx\dt - \int_0^T\int_\Omega c_iu_i \eta_\sigma \cdot \frac{\nabla c_i}{c_i+\delta} \,\dx\dt \\
    & \phantom{xx} = \int_0^T\int_\Omega [(c_iu_i)^\epsilon - c_iu_i] \eta_\sigma \cdot \frac{\nabla c_i^\epsilon}{c_i^\epsilon+\delta} \,\dx\dt + \int_0^T\int_\Omega c_iu_i \eta_\sigma \cdot \frac{\nabla c_i^\epsilon - \nabla c_i}{c_i^\epsilon+\delta} \,\dx\dt \\
    & \phantom{xxxx} + \int_0^T\int_\Omega c_iu_i \eta_\sigma \cdot \nabla c_i \left(\frac{1}{c_i^\epsilon+\delta}-\frac{1}{c_i+\delta}\right) \,\dx\dt,
\end{split} 
\] 

\noindent
and notice that $c_iu_i$ and $\nabla c_i$ are both in $L^2(0,T;L^2(\Omega))$, due to $\sqrt{c_i}u_i$ and $\nabla\sqrt{c_i}$ being in $L^2(0,T;L^2(\Omega))$ and $c_i\in L^\infty(0,T;L^\infty(\Omega))$. This implies that $(c_iu_i)^\epsilon \to c_iu_i$ and $\nabla c_i^\epsilon = (\nabla c_i)^\epsilon \to \nabla c_i$ as $\epsilon\to0$, both in $L^2(0,T;L^2(\Omega))$. Then, due to the bounds $|\eta_\sigma(t)|\leq1$ and $\frac{1}{c_i^\epsilon+\delta}\leq \frac{1}{\delta}$, we have

\[ \begin{split}
    \left| \int_0^T\int_\Omega [(c_iu_i)^\epsilon - c_iu_i] \eta_\sigma \cdot \frac{\nabla c_i^\epsilon}{c_i^\epsilon+\delta} \,\dx\dt \right| & \leq \frac{1}{\delta}\int_0^T\int_\Omega |(c_iu_i)^\epsilon - c_iu_i| |\nabla c_i^\epsilon| \,\dx\dt \\
    & \leq \frac{1}{\delta} \|(c_iu_i)^\epsilon - c_iu_i\|_{L^2(0,T;L^2(\Omega))}\|\nabla c_i\|_{L^2(0,T;L^2(\Omega))} \\
    & \to 0, \textnormal{ as $\epsilon\to0$,}
\end{split} \]

\noindent
and similarly
\[ \begin{split}
     \left| \int_0^T\int_\Omega c_iu_i \eta_\sigma \cdot \frac{\nabla c_i^\epsilon - \nabla c_i}{c_i^\epsilon+\delta} \,\dx\dt \right| & \leq \frac{1}{\delta}\int_0^T\int_\Omega |c_iu_i||\nabla c_i^\epsilon-\nabla c_i| \,\dx\dt \\
     & \leq \frac{1}{\delta}\|c_iu_i\|_{L^2(0,T;L^2(\Omega))}\|(\nabla c_i)^\epsilon - \nabla c_i\|_{L^2(0,T;L^2(\Omega))} \\
     & \to 0, \textnormal{ as $\epsilon\to0$}.
\end{split} \]

For the third term, we apply the Dominated Convergence Theorem, since
\[ 
c_iu_i \eta_\sigma \cdot \nabla c_i \left(\frac{1}{c_i^\epsilon+\delta}-\frac{1}{c_i+\delta}\right) \to 0 \textnormal{ a.e. in $\Omega\times(0,T)$, as $\epsilon\to0$,} 
\]
\noindent
and
\[ 
\begin{split}
    \left| c_iu_i \eta_\sigma \cdot \nabla c_i \left(\frac{1}{c_i^\epsilon+\delta}-\frac{1}{c_i+\delta}\right) \right| & = |c_iu_i||\eta_\sigma||\nabla c_i|\frac{|c_i-c_i^\epsilon|}{|(c_i+\delta)(c_i^\epsilon+\delta)|} \\
    & \leq \frac{2}{\delta^2} |c_iu_i| |\nabla c_i| \in L^1(0,T;L^1(\Omega)),
\end{split}
\]

\noindent
thus establishing \[ \int_0^T\int_\Omega (c_iu_i)^\epsilon \eta_\sigma \cdot \frac{\nabla c_i^\epsilon}{c_i^\epsilon+\delta} \,\dx\dt \to \int_0^T\int_\Omega c_iu_i \eta_\sigma \cdot \frac{\nabla c_i}{c_i+\delta} \,\dx\dt. \]

Therefore, letting $\epsilon\to 0^+$ in \eqref{(1)}, we obtain \begin{equation}
    \label{(2)} \begin{split}
        & \frac{1}{\sigma}\int_\sigma^{2\sigma}\int_\Omega (c_i+\delta)\ln{(c_i+\delta)} \,\dx\dt - \frac{1}{\sigma}\int_{T-2\sigma}^{T-\sigma}\int_\Omega (c_i+\delta)\ln{(c_i+\delta)} \,\dx\dt \\
        & \phantom{x} - \frac{1}{\sigma}\int_\sigma^{2\sigma}\int_\Omega c_i \,\dx\dt + \frac{1}{\sigma}\int_{T-2\sigma}^{T-\sigma}\int_\Omega c_i \,\dx\dt + \int_0^T\int_\Omega c_iu_i \eta_\sigma(t) \cdot \frac{\nabla c_i}{c_i+\delta} \,\dx\dt = 0.
    \end{split}
\end{equation}

\noindent\underline{Step 2}: Now keeping $\delta>0$ fixed, let $\sigma\to0^+$. We start with the third term and see that \[ \frac{1}{\sigma}\int_\sigma^{2\sigma}\int_\Omega c_i \,\dx\dt \to \int_\Omega c_i^0 \,\dx. \]

\noindent
The reason is that
\begin{align*}
    \left| \frac{1}{\sigma}\int_\sigma^{2\sigma}\int_\Omega c_i \,\dx\dt - \int_\Omega c_i^0 \,\dx \right| & = \left| \frac{1}{\sigma}\int_\sigma^{2\sigma}\int_\Omega (c_i - c_i^0) \,\dx \dt\right| \\
    & \leq \frac{1}{\sigma}\int_\sigma^{2\sigma}\int_\Omega |c_i - c_i^0| \,\dx \dt \\
    & = \frac{1}{\sigma}\int_\sigma^{2\sigma}\|c_i - c_i^0\|_{L^1(\Omega)} \,\dt, 
\end{align*}
which converges to zero as $\sigma\to0$ since the integrand is continuous due to $c_i\in C([0,T];L^2(\Omega))$. The fourth term is treated the exact same way.

For the first term of \eqref{(2)}, we use the Mean Value Theorem to write: 
\begin{align*}
    & \left| \frac{1}{\sigma}\int_\sigma^{2\sigma}\int_\Omega (c_i+\delta)\ln{(c_i+\delta)} \,\dx\dt - \int_\Omega (c_i^0+\delta)\ln{(c_i^0+\delta)} \,\dx \right| \\
    & \phantom{xxxxxx} \leq \frac{1}{\sigma}\int_\sigma^{2\sigma}\int_\Omega \left| (c_i+\delta)\ln{(c_i+\delta)} - (c_i^0+\delta)\ln{(c_i^0+\delta)} \right| \,\dx \dt \\
    & \phantom{xxxxxx} = \frac{1}{\sigma}\int_\sigma^{2\sigma}\int_\Omega |\ln{(\xi+\delta)}+1||c_i - c_i^0| \,\dx \dt \\
    & \phantom{xxxxxx} \leq M(\delta)\frac{1}{\sigma}\int_\sigma^{2\sigma} \|c_i - c_i^0\|_{L^1(\Omega)} \,\dx \dt,
\end{align*} where $\xi$ lies between $c_i$ and $c_i^0$ and thus $0<\xi<1$ and $M(\delta)$ is a positive constant which only depends on $\delta$. Then, taking the limit $\sigma\to0^+$ we conclude that \[ \frac{1}{\sigma}\int_\sigma^{2\sigma}\int_\Omega (c_i+\delta)\ln{(c_i+\delta)} \,\dx\dt \to \int_\Omega (c_i^0+\delta)\ln{(c_i^0+\delta)} \,\dx, \] and similarly for the second term of \eqref{(2)}.

For the last term of \eqref{(2)}, since $\eta_\sigma(t) \to 1$ almost everywhere in $(0,T)$ we see that \[ c_iu_i \eta_\sigma(t) \cdot \frac{\nabla c_i}{c_i+\delta} \to c_iu_i \cdot \frac{\nabla c_i}{c_i+\delta} ~ \textnormal{almost everywhere in } \Omega\times(0,T), \]
\noindent
and again due to the bounds $|\eta_\sigma(t)|\leq1$ and $\frac{1}{c_i^\epsilon+\delta}\leq \frac{1}{\delta}$ and the fact that $c_iu_i$ and $\nabla c_i$ are in $L^2(0,T;L^2(\Omega))$ we obtain \[ \left|c_iu_i \eta_\sigma(t) \cdot \frac{\nabla c_i}{c_i+\delta}\right| \leq \frac{1}{\delta}|c_iu_i||\nabla c_i| \in L^1(0,T;L^1(\Omega)). \] and Lebesgue's Dominated Convergence Theorem can be applied.

Then, passing to the limit $\sigma\to0^+$ in \eqref{(2)} gives \begin{equation}
    \label{(3)} \begin{split}
        & \int_\Omega (c_i^0+\delta)\ln{(c_i^0+\delta)} \,\dx - \int_\Omega (c_i(T)+\delta)\ln{(c_i(T)+\delta)} \,\dx \\
        & \phantom{xx} - \int_\Omega c_i^0 \,\dx + \int_\Omega c_i(T) \,\dx + \int_0^T\int_\Omega c_iu_i \cdot \frac{\nabla c_i}{c_i+\delta} \,\dx\dt = 0  .
    \end{split}
\end{equation}

\noindent\underline{Step 3}: We remove $\delta$, needed to regularize the logarithm, by letting $\delta\to0^+$. The first term gives 
\[ 
\int_\Omega (c_i^0+\delta)\ln(c_i^0+\delta) \,\dx \to\int_\Omega c_i^0\ln c_i^0 \,\dx,
\]
by the Dominated Convergence Theorem, since $(c_i^0+\delta)\ln(c_i^0+\delta)\to c_i^0\ln c_i^0$ almost everywhere as $\delta\to0^+$ and $|(c_i^0+\delta)\ln(c_i^0+\delta)|\leq C$ uniformly in $\delta$ for some constant $C>0$ independent of $\delta$, since $0\leq c_i^0 \leq 1$ and $0<\delta<1$ and constants are integrable in $\Omega$. The second term is treated similarly.

Finally, for the last term we obtain 
\[ \int_0^T\int_\Omega c_iu_i\cdot\frac{\nabla c_i}{c_i+\delta} \,\dx \dt\to 2\int_0^T\int_\Omega m_i\cdot\nabla\sqrt{c_i} \,\dx \dt, 
\] 
by the Dominated Convergence Theorem, since for $m_i=\sqrt{c_i}u_i$
\[ c_iu_i\cdot\frac{\nabla c_i}{c_i+\delta}\to 2m_i\cdot\nabla \sqrt{c_i} \quad \textnormal{ almost everywhere as $\delta\to0^+$}, 
\] and 
\[ 
\begin{split}
    \left|c_iu_i\cdot\frac{\nabla c_i}{c_i+\delta}\right| & = |m_i||2\nabla\sqrt{c_i}|\frac{c_i}{c_i+\delta} \\
    & \leq |m_i||2\nabla\sqrt{c_i}|\in L^1(0,T;L^1(\Omega)),
\end{split} 
\] 
since by the regularity of weak solutions $m_i$ and $\nabla\sqrt{c_i}$ are in $L^2(0,T;L^2(\Omega))$ and the bound is uniform in $\delta$.

Therefore, letting $\delta\to0^+$ we obtain the entropy equality \begin{equation} \label{eneq}
\int_\Omega c_i(T)(\ln c_i(T)-1) \,\dx -2\int_0^T\int_\Omega m_i\cdot\nabla \sqrt{c_i} \,\dx \dt=\int_\Omega c_i^0(\ln c_i^0-1) \,\dx.
\end{equation} 

We sum \eqref{eneq} over all $i\in\{1,\dots,n\}$ and use the definition of the entropy \eqref{entropy}, to arrive at

\[ H(\cc(T)) - 2\sum_{i=1}^n\int_0^T\int_\Omega m_i\cdot\nabla \sqrt{c_i} \,\dx \dt = H(\cc^0). \]

Performing the following computation, which derives from 
\eqref{alg-system} and the symmetry of $D_{ij}$
\[ \begin{split}
    - \sum_{i=1}^n2m_i\cdot\nabla \sqrt{c_i} & = \sum_{i=1}^n\sum_{j=1}^nA_{ij}m_jm_i \\
    & = - \sum_{i=1}^n\sum_{j\ne i}\frac{1}{D_{ij}}c_iu_i\cdot c_ju_j + \sum_{i=1}^n\sum_{j\ne i}\frac{1}{D_{ij}}c_jm_i^2 \\
    & = \frac{1}{2}\sum_{i=1}^n\sum_{j\ne i}\frac{c_ic_j}{D_{ij}}|u_i-u_j|^2,
\end{split} \] we obtain \begin{equation}
    \label{final} H(\cc(T)) + \frac{1}{2} \int_0^T\int_\Omega \sum_{i=1}^n\sum_{j\ne i}\frac{c_ic_j}{D_{ij}}|u_i-u_j|^2 \,\dx \dt = H(\cc^0).
\end{equation}

\noindent
The fact that $T>0$ is arbitrary, implies that \eqref{final} holds for all times $t\in(0,T]$ and the proof is complete.
\end{proof}

\section*{Acknowledgments}
LCB acknowledges support by INdAM GNAMPA and also by MIUR,
within project PRIN20204NT8W4$_{-}$004: Nonlinear evolution PDEs,
fluid dynamics and transport equations: theoretical foundations and applications. LCB also thanks King Abdullah University of
Science and Technology (KAUST) for the support and hospitality during the preparation of the paper. The research of SG and AET was partially supported by King Abdullah University of Science and Technology (KAUST) baseline funds. SG also acknowledges partial support from the Austrian Science Fund (FWF), grants P33010 and F65. This work has received funding from the European Research Council (ERC) under the European Union's Horizon 2020 research and innovation programme, ERC Advanced Grant no. 101018153.

\end{document}